\documentclass[11pt, a4paper]{amsart}

\usepackage{amsthm} 				
\usepackage{amsmath}				
\usepackage{amssymb}				
\usepackage{amsfonts}

\usepackage{csquotes}				

\usepackage[T1]{fontenc} 				

\usepackage[a4paper, total={16.5cm, 23.2cm}]{geometry}				

\usepackage{tikz-cd}					
\usepackage[all]{xy}					
\usepackage{pgf,tikz}
\usepackage{tikzsymbols}
\usepackage{pgf,tikz}
\usepackage{tikz-cd}
\usepackage{graphicx}
\usetikzlibrary{shapes}
\usetikzlibrary{cd}
\usetikzlibrary{arrows}
\usetikzlibrary{topaths}
\usetikzlibrary{decorations.pathmorphing}
\usetikzlibrary{shadows.blur}

\usetikzlibrary{calc}
\usepackage{mathrsfs}
\usepackage{amsfonts}
\usetikzlibrary{arrows}
\usepackage[mathscr]{eucal}
\usepackage{ulem}

\usepackage{graphicx}
\usepackage{caption}
\usepackage{subcaption}
\usepackage{hyperref}

\newcommand{\End}{\mbox{End}}
\newcommand{\Hom}{\mbox{Hom}}

\newcommand{\Fac}{\mbox{Fac}}

\newcommand{\T}{\mathcal{T}}

\renewcommand{\mod}{\mbox{mod}}
\renewcommand{\top}{\mbox{top}}
\newcommand{\soc}{\mbox{soc}}

\newtheorem{theorem}{Theorem}[section]
\newtheorem{corollary}[theorem]{Corollary}
\newtheorem{lemma}[theorem]{Lemma}
\newtheorem{proposition}[theorem]{Proposition}
\newtheorem{conjecture}[]{Conjecture}

\usepackage[color=blue!20!white]{todonotes} 

\theoremstyle{definition}

\theoremstyle{remark}
\newtheorem{remark}[theorem]{Remark}
\newtheorem{example}[theorem]{Example}

\begin{document}
\normalem

\title{On band modules and $\tau$-tilting finiteness}

\author{Sibylle Schroll}

\address{School of Mathematics and Actuarial Sciences, University of Leicester, Leicester, LE1  7RH}
\email{schroll@leicester.ac.uk and  yvd1@leicester.ac.uk}
\address{Hausdorff Center of Mathematics, Rheinische Friedrich-Wilhelms-Universit\"at Bonn, Bonn, Germany}
\email{hipolito.treffinger@hcm.uni-bonn.de}
\thanks{The first and the second author are supported by the EPSRC through the Early Career Fellowship, EP/P016294/1. The first and the third author are supported by the Royal Society through the Newton International Fellowship NIF\textbackslash R1\textbackslash 180959.
The second author is also funded by the Deutsche Forschungsgemeinschaft (DFG, German Research Foundation) under 
Germany's Excellence - EXC-2047/1-390685813.
}

\author{Hipolito Treffinger}

\author{Yadira Valdivieso}






\begin{abstract}
In this paper, motivated by a $\tau$-tilting version of the Brauer-Thrall Conjectures, we study general properties of band modules and their endomorphisms in the module category of a finite dimensional algebra. 
As an application we describe properties of torsion classes containing band modules.
Furthermore, we show that a special biserial algebra is $\tau$-tilting finite if and only if no band module is a brick. We also recover a criterion for the $\tau$-tilting finiteness of Brauer graph algebras in terms of the Brauer graph.
\end{abstract}

\maketitle

\section{Introduction}

An important breakthrough in the representation theory of finite-dimensional algebras was the systematic introduction of quivers, a powerful tool bringing linear algebra into the theory. 
In particular, every finite-dimensional algebra over an algebraically closed field $K$ is Morita equivalent to a quotient $KQ/I$ of a path algebra $KQ$ of a quiver $Q$ by an admissible ideal $I$ \cite{Gabriel1972}. 
Presentations of algebras in terms of quivers with relations encode much  homological and geometric information.
As a consequence, quiver representations now play an important role in many areas of mathematics such as, for example, algebraic geometry, mathematical physics, and mirror symmetry.

Since the introduction of quivers to representation theory, much work  has been done to show that many families of algebras defined by homological properties can, in fact, be characterised by properties of their quivers and relations. 

One important such family of algebras is that of special biserial algebras. 
This class contains many well-known families of algebras, such as gentle algebras which play a central role in cluster theory \cite{Assem2010} and homological mirror symmetry of surfaces \cite{HKK} and Brauer graph algebras which originate in the modular representation theory of finite groups \cite{Dade1966}.
For these algebras much of the representation theory is encoded in the combinatorics of their quivers and relations. 
For example,  the isomorphism classes of finitely generated  indecomposable modules are given by certain words in the alphabet consisting of the arrows and formal inverses of arrows of their quivers \cite{BR87, WW85}.
This naturally divides the indecomposable modules over these algebras into two classes, the so-called \emph{string modules} and infinite families of \textit{band  modules}. 
The morphisms between string and band modules can also be described in terms of word combinatorics in the quiver \cite{CB89,Kra91}.   

One of the motivating observations of this paper is that for any finite-dimensional algebra given by quiver and relations, that is algebras which are not necessarily special biserial, we can still consider the combinatorics of string and band modules. 
We show that in the general case, these modules still encode significant information on the module categories of the algebras.
Even though, in this case there are (possibly infinitely many) indecomposable modules that cannot be described in terms of string and band modules.

More recently, the theory of cluster algebras has given new impetus to representation theory with the introduction of many new cluster-inspired representation theoretic concepts. 
An excellent example of this is the introduction of $\tau$-tilting theory, inspired by  mutation in cluster algebras\;\cite{AIR}. 
Since its introduction in 2014, $\tau$-tilting theory has been intensively studied and led to the introduction of many new concepts promising to relate representation theory with other areas of mathematics such as Hall algebras, Donladson-Thomas invariants  and Riemannian geometry.

From a representation theoretic point of view, one of the important results in \cite{AIR} are explicit bijections between functorially finite torsion classes, support $\tau$-tilting modules and $2$-term silting complexes in the bounded derived category of a finite dimensional algebra. 
Accordingly an algebra is called \emph{$\tau$-tilting finite} if has finitely many support $\tau$-tilting modules. 
Furthermore, for a $\tau$-tilting finite algebra, there are only finitely many torsion classes and all  are functorially finite\;\cite{DIJ19}.
A further characterisation of $\tau$-tilting finite algebras is via \emph{bricks} in their module category. 
An object in the module category of an algebra is called a \textit{brick} if its endomorphism algebra is a division ring.
By \cite{DIJ19}, an algebra is $\tau$-tilting  finite if and only if there are finitely many bricks in its module category.
This immediately implies that if the module category of an algebra contains a band module which is a brick then the algebra is $\tau$-tilting infinite (see Proposition~\ref{prop:bandbrick}). 

The representation theory of a $\tau$-tilting finite algebra is considerably easier to understand than that of a $\tau$-tilting infinite algebra.  
For example, the support of the scattering diagram of a $\tau$-tilting finite algebra  is completely determined by its support $\tau$-tilting modules \cite{Bridgeland2017, BST2019} 
and  the stability manifold of a finite dimensional algebra $A$ is contractible if the algebra is silting-discrete which implies, in particular, that the heart of any bounded t-structure of the derived category of $A$  is a module category over a $\tau$-tilting finite algebra \cite{PSZ18}.

Therefore, a classification of $\tau$-tilting finite algebras is almost as important in today's representation theory as the determination of representation finite algebras was in the last century.

Much of the motivation in the early days of representation theory of finite dimensional algebras stems from this quest of determining algebras of finite representation type with an important role being played by the first and second Brauer-Thrall Conjectures, which were proved subsequently by Ro\u{\i}ter \cite{Roiter1968} and Auslander \cite{Auslander1974} for the first and by Nazarova and Ro\u{\i}ter \cite{Nazarova1971} and Bautista \cite{Bautista1985} for second.

We now state a $\tau$-tilting analogue of the first and second Brauer-Thrall Conjectures.
\begin{conjecture}[First $\tau$-Brauer-Thrall Conjecture]\label{conj:tBT1}
Let $A$ be a finite dimensional algebra over a field $K$. 
If there exists a positive integer $n$ such that $\dim_K M \leq n$ for every finite dimensional $A$-module which is a brick, then $A$ is $\tau$-tilting finite.
\end{conjecture}

Conjecture~\ref{conj:tBT1} immediately implies the first Brauer-Thrall Conjecture in the case of a $\tau$-tilting infinite algebra. We note that this conjecture has also been stated in \cite[Conjecture 6.6]{Mou19} and it has been shown to hold for any finite dimensional algebra over any field in \cite{ST2020}.

\begin{conjecture}[Second $\tau$-Brauer-Thrall Conjecture]\label{conj:tBT2}
Let $A$ be a finite dimensional algebra over a field $K$. 
Suppose that $A$ is $\tau$-tilting infinite. 
Then there is a positive integer $d$ such that there are infinitely many bricks $M$ where $\dim_K M = d$.
\end{conjecture}

We note that the second $\tau$-Brauer-Thrall Conjecture is not a direct translation of the second Brauer-Thrall Conjecture. 
A direct translation of the latter to a $\tau$-tilting version would state that if an algebra is $\tau$-tilting infinite then there exist infinitely many integer vectors $d$ such that there exist infinitely many isomorphism classes of bricks of dimension vector $d$. 
Clearly this does already not hold in the case of the Kronecker algebra. 
However, even though the formulation of Conjecture \ref{conj:tBT2} is not a direct generalisation of the classical second Brauer-Thrall Conjecture, if Conjecture \ref{conj:tBT2} holds for a $\tau$-tilting infinite algebra $A$ then it follows from \cite{Smalo1980} that the classical second Brauer-Thrall Conjecture holds for $A$.

Note that the choice of the name `$\tau$-Brauer-Thrall Conjectures' is guided by the fact that the conjectures  are concerned with the $\tau$-tilting finiteness of the algebras. 

We show by exploiting the interplay between the combinatorics of bricks and band modules that both the first and the second $\tau$-Brauer-Thrall Conjectures hold for special biserial algebras over an algebraically closed field.

More precisely, we show the following.

\begin{theorem}[Theorem \ref{thm:specialbis}]
Let $A=KQ/I$ be a special biserial algebra over an algebraically closed field $K$.
Then $A$ is $\tau$-tilting finite if and only if no  band module of $A$ is a brick.
\end{theorem}

\begin{corollary}
Conjectures \ref{conj:tBT1} and \ref{conj:tBT2} hold for special biserial algebras.
\end{corollary}

In order to prove Theorem 1.1 we heavily rely on the following property of endomorphisms of band modules.
For notation see Section~\ref{sec:background}.

\begin{theorem}[Theorem~\ref{thm:lengthgraphmap}]
Let $A$ be a finite dimensional algebra over an algebraically closed field $K$ and let $b$ be a band.
Suppose that there exists a non-trivial nilpotent endomorphism $f_w : M(b,\lambda, 1) \to M(b, \lambda, 1)$  induced by a subword $w$ of $^\infty b ^\infty$.
Then $w$ is a proper subword of some rotation of $b$.
\end{theorem}
 
As a consequence of Theorem 1.3, we show the following results on torsion classes based on the band modules they contain. 

\begin{theorem}[Theorem~\ref{thm:torsion}]
Let $K$ be an algebraically closed field and let $A$ be a finite dimensional $K$-algebra  containing a band module $M(b, \lambda, 1)$, for some $\lambda \in K^*$. 
\begin{enumerate}
    \item If $M(b, \lambda, 1) \in \T$ for some torsion class $\T$, then $M(b, \lambda, n) \in \T$, for all $n \in \mathbb{N}$.
    \item If $M(b, \lambda, 1)$ is not a brick  and $M(b, \lambda, 1)$ is in  some torsion class $\T$, then $M(b, \mu, 1) \in \T$ for all $\mu \in K^*$.
    \item If $M(b, \lambda, 1)$ is a brick then there exist infinitely many distinct torsion classes $\T_\mu$, for $\mu \in K^*$, such that $\T_\mu$ contains $M(b, \eta, 1)$, for $\eta \in K^*$, if and only if $\eta = \mu$.
\end{enumerate}
\end{theorem}

A different classification of $\tau$-tilting finite special biserial algebras has recently been obtained in \cite{Mousavand2019} based on the classification of minimal representation infinite special biserial algebras in \cite{Ringel2011}. 
Gentle algebras are a subclass of special biserial algebras. 
For these algebras a similar criterion to the one obtained in this paper was given in \cite{Plamondon2018}.
 
We finish the paper in Section~\ref{sec:BGA} with an application of our criterion to obtain a characterisation of $\tau$-tilting finite Brauer graph algebras, giving a new proof of \cite[Theorem 6.7]{AAC18}.

\medskip 

{\bf Acknowledgements:} 
For helpful conversations regarding band modules and their morhpisms, the authors would like to thank Rosanna Laking and Jan Schr\"oer. Particular thanks goes to the latter for suggesting to formulate our results in terms of a $\tau$-tilting version of the Brauer-Thrall Conjectures.  The authors also  thank Alexandra Zvonareva for her  comments, and the referee for pointing out a mistake in an earlier version of the paper. 

\section{Background}\label{sec:background}
In this section we fix  some  notation and definitions which will be used throughout the paper.

We fix an algebraically closed field $K$, and $A$ a  finite dimensional $K$-algebra, which is Morita equivalent to $KQ/I$ for some finite quiver $Q$ and admissible ideal $I$ in $KQ$ (see \cite{Gabriel1972}). 
We call elements of a generating set of $I$ \emph{relations}. 
Note that we use the same notation for elements in $KQ$ and elements in $KQ/I$ with the implicit understanding that the latter are representatives in their equivalence class.

Given an $A$-module $M$, we call the \textit{top} of $M$, denoted $\top M$, the largest semisimple quotient of $M$.
Similarly, we call the \textit{socle} of $M$, denoted $\soc M$, the largest semisimple submodule of $M$. 

For a quiver  $Q$, let $Q_0$ be the set of vertices of $Q$ and  $Q_1$ the set of arrows. 
If $\alpha$ is an arrow of $Q$, we denote by $s(\alpha)\in Q_0$ and by $t(\alpha) \in Q_0$ the source and target point of $\alpha$, respectively.

For every arrow $\alpha:i \to j$, we define $\bar{\alpha}:j\to i$ to be its formal inverse. 
Let $\overline{Q_1}$ be the set of formal inverses of the elements of $Q_1$. We refer to the elements of $Q_1$ as direct arrows and to the elements of $\overline{Q_1}$ as inverse arrows.
In general, by abuse of notation, we assume that $\overline{\overline{\alpha}} = \alpha$.
A \emph{walk} is a sequence $\alpha_1\dots \alpha_n$ of elements of $Q_1\cup \overline{Q_1}$ such that $t(\alpha_i)=s(\alpha_{i+1})$ for every $i=1, \dots, n-1$ and such that $\alpha_{i+1} \neq \overline{\alpha_i}$.
We say that a walk is of \textit{length} $t$ if it is the composition of exactly $t$ elements in $Q_1 \cup \overline{Q_1}$.

We recall that a \emph{string} in $A$ is by definition a walk $w$ in $Q$ such that no subpath $w'$ of $w$ or subpath $\overline{w'}$ of $\overline{w}$ is a summand in a relation in $I$. 
A \emph{band} $b$ is defined to be a cyclic string such that every power $b^n$ is a string, but $b$ itself is not a proper power of some string.
A string $w=\alpha_1\dots \alpha_n$ is a \emph{direct} (respectively, \emph{inverse}) if $\alpha_i$ is a direct arrow (respectively, inverse arrow) for every $i=1, \dots, n$.

Given a string $w$, the \emph{string module} $M(w)$ corresponds to the quiver representation induced by replacing each vertex in $w$ by a copy of the field $K$ and every arrow in $w$ by the identity map. 
In a similar way, given a band $b=\alpha_1\dots \alpha_t$, a non-zero element $\lambda$ in $K^*$ and $n\in\mathbb N$, the band module $M(b, \lambda, n)$ is obtained from the band $b$ by replacing each vertex by a copy of the $K$-vector space $K^n$ and every arrow $\alpha_i$ for $1\leq i < t$ by the identity matrix of dimension $n$, and $\alpha_t$ by a Jordan block of dimension $n$ and eigenvalue $\lambda$ (we refer \cite{BR87} for the precise definition).

\begin{remark}
If $A$ is not special biserial and $b$ is a band then the induced band module $M(b, \lambda, n)$ might not lie in a homogenous tube of rank $1$.
For example, this is the case for any band in the $3$-Kronecker algebra.
\end{remark}

Let $w = \alpha_1  \ldots \alpha_i \ldots \alpha_j \ldots \alpha_t$ be a 
string and let $u = \alpha_i \ldots \alpha_j$.
Then we say that $u$ is a \emph{submodule substring} if  $\alpha_{i-1}$ is direct and $\alpha_{j+1}$ is inverse, noting that if $i$ is equal to 1 or if $j$ is  equal to $t$, we still consider $u$ a submodule substring. 
We say that $u$ is a \emph{quotient substring} if $\alpha_{i-1}$ is inverse and $\alpha_{j+1}$ is direct, and noting again that if $i$ is equal to 1 or if $j$ is  equal to $t$, we still consider $u$ a quotient  substring. 

Given two strings $v, w$ such that they have a common substring $u$ which is a quotient substring of $v$ and a submodule substring of $w$, then by \cite{CB89} there is a non-zero map from $M(v)$ to $M(w)$ and the maps of this form give a basis of $\Hom_A(M(v), M(w))$. 

Maps between bands are slightly different from maps between strings \cite{Kra91}, see also, for example, \cite{LakingPhD}. 
For completeness we recall the construction of a basis morphism between bands. 
We extend the definition of strings to infinite strings, noting that the only  infinite strings  we will be considering are the strings  of the form $^\infty b^ \infty$  which are the infinite strings formed by infinitely many compositions of a band $b$ with itself.

Let $b$ and $c$ be two bands, $\lambda, \mu \in K^*$ and $n, m$ two positive integers. 
If $b$ is different from $c$ or $\lambda$ is different from $\mu$, a basis of $\Hom_A(M(b,\lambda, n), M(c,\mu, m))$ is given by maps $f_{(w,\phi)}$ induced by pairs $(w, \phi)$, where $w$ is a string of finite length which is a quotient substring of $^\infty b^ \infty$ and a submodule substring of $^\infty c^\infty$ and $\phi$ is an element of a given basis of $\Hom_K(K^n, K^m)$.

If $b=c$ and $\lambda=\mu$, a basis of $\Hom_A(M(b,\lambda, n), M(b,\lambda, m))$ is given by the maps induced by the pairs $(w, \phi)$ as above and  maps $f_\psi$  induced by $K$-linear maps $\psi\in \Hom_K(K^n, K^m)$ such that for every vertex $i$ in $Q$,  
$(f_\psi)_i\-:(M(b, \lambda, n))_i \to (M(b, \lambda, m))_i$ is equal to a diagonal by blocks $nk\times mk$-matrix, where the diagonal blocks correspond to the maps $\psi$ and where $k$ is the number of times the vertex $i$ is the start of a letter in $b$.
We note, in particular, that the maps $f_\psi$ do not correspond to substrings of $^\infty b ^\infty$ which are at the same time submodule strings and quotients strings.
In the case $m=n=1$, the only non-zero basis element of the form $f_\psi$ is the identity map.

\begin{example}\label{ex:morphism}
Let $A$ be the algebra given by the quiver
$
\begin{tikzcd}
\bullet \arrow[r, shift right, "\beta" below] \arrow[r, shift left, "\alpha" above] & \bullet \arrow[r, shift right, "\delta" below] \arrow[r, shift left, "\gamma" above] & \bullet
\end{tikzcd}
$
and the ideal $I= \langle \alpha\delta, \beta\gamma \rangle$.
Observe that  $b=\gamma\overline{\delta}$  and $c=\alpha\gamma\overline{\delta}\gamma\overline{\delta}\ \overline{\beta}$ are bands in $A$ and that $\gamma\overline{\delta}\gamma\overline{\delta}$ is a quotient string of $^\infty b^\infty$ and a submodule substring of $^\infty c^\infty$. 
Let $M(b, \lambda, n)$ and $M(c, \mu, m)$ be  two band modules associated to $b$ and $c$ respectively, with $\lambda, \mu \in K^*$ and $n,m\in \mathbb N$. 
Then, for any morphism $\phi \in \Hom_K(K^n, K^m)$,  the pair $(\gamma\overline{\delta}\gamma\overline{\delta}, \phi)$ gives rise to a basis element $f_{(\gamma\overline{\delta}\gamma\overline{\delta}, \phi)}  \in \Hom_A(M(b,\lambda, n), M(c,\mu, m))$ induced by the diagram in Figure~1 with the following notation:
$U=K^n$ and $V=K^m$ and $\Psi$ is the $(n \times n)$-Jordan block of  eigenvalue $\lambda$ and $\Phi$ is the $(m\times m)$-Jordan block of eigenvalue $\mu$.

\begin{figure}[ht!]
    \begin{tikzcd}
    \cdots \ar[dr] & & U \arrow[dl, "1_K" above, "\delta"] \arrow[dr, "\Psi" above, "\gamma" below] \arrow[ddd, dashed, "\phi"] & & U \arrow[dl, "1_K" above, "\delta"] \arrow[dr, "\Psi" above, "\gamma"below] \arrow[ddd, dashed, "\phi\Psi ^{-1}"] & & U \arrow[dl, "1_K" above, "\delta"] \arrow[dr, "\Psi" above, "\gamma"below] \arrow[ddd, dashed, "\phi\Psi ^{-2}"] & & \cdots \ar[dl]\\
     & U& & U \arrow[ddd, dashed, "\phi\Psi ^{-1}"] &  & U \arrow[ddd, dashed, "\phi\Psi ^{-2}"]  &  & U& \\
     & V\arrow[dr, "1_K" above, "\alpha" below] \ar[dl] & & &  & & & V \arrow[dl, "\Phi" above, "\beta"]\ar[dr]&  \\
  \cdots & & V \arrow[dr, "1_K" above, "\gamma" below] & & V \arrow[dl, "1_K" above, "\delta"] \arrow[dr, "1_K" above, "\gamma"below] & & V \arrow[dl, "1_K" above, "\delta"]  & & \cdots\\
    & & & V&  &V & & &
  \end{tikzcd}
  \caption{Diagram of $(\gamma\overline{\delta}\gamma\overline{\delta},\phi)$}
  \label{fig:graphmaps}
\end{figure}

Then the morphism $f_{(\gamma\overline{\delta}\gamma\overline{\delta},\phi)}: M(b, \lambda, n) \to M(c, \mu, m)$ is as follows.

\[
\begin{tikzcd}[column sep=2cm, ampersand replacement=\&]
0 \arrow[r, shift right, "0" below] \arrow[r, shift left, "0" above] \arrow[dd, "A"] \& U \arrow[r, shift right, "1_k" below] \arrow[r, shift left, "\Psi" above] \arrow[dd, "B"] \& U \arrow[dd, "C"]\\
\& \& \\
V \arrow[r, shift right, "\begin{bmatrix}0 \, \Phi\,  0\end{bmatrix}^{T}" below] \arrow[r, shift left, "\begin{bmatrix}0 \,  1_K\,  0 \end{bmatrix}^{T}" above] \& V^3 \arrow[r, shift right, "\begin{bmatrix}0  1  0\\  0  0  1\end{bmatrix}" below] \arrow[r, shift left, "\begin{bmatrix}1  0  0\\ 0  1  0 \end{bmatrix}" above] \& V^2
\end{tikzcd}
\]
where $A=\begin{bmatrix}0\end{bmatrix}$, $B=\begin{bmatrix}\phi\\ \phi\Psi^{-1}\\ \phi\Psi^{-2}\end{bmatrix}$, and $C=\begin{bmatrix} \phi\Psi^{-1}\\ \phi\Psi^{-2}\end{bmatrix}$.
\end{example}

\begin{remark}\label{rmk:n=1}
Let $b$ be a band in $A$. 
Then every non-zero non-trivial endomorphism of the band module $M(b,\lambda, 1)$ is a linear combination of maps that are determined by pairs $(w, Id_K)$, where $w$ is a string of finite length which is at the same time a quotient substring and a submodule substring of $^\infty b^ \infty$.
For ease of notation, we denote the map $f_{(w, Id_K)}$ by $f_w$.
\end{remark}

\bigskip
\section{Band modules and their endomorphisms}\label{sec:bands}

In this section, we begin by recalling   that if an algebra $A$ contains a band module which is a brick then $A$ is $\tau$-tilting infinite. This is a direct consequence of  \cite[Theorem 1.4]{DIJ19}.  

\begin{proposition}\label{prop:bandbrick}
Let $A= KQ/I$ be a finite dimensional algebra. 
If there exists a band module $M$ which is a brick, then $A$ is $\tau$-tilting infinite.
\end{proposition}

\begin{proof}
Since $K$ is algebraically closed and hence infinite, for any band $b$ there is an infinite family $\{M(b,\lambda,n) : \lambda \in K^*, n \in \mathbb{N}\}$ of non-isomorphic indecomposable modules.

By hypothesis, there exists a band $b$, $\lambda \in K^*$, and $n \in \mathbb{N}$ such that $M(b,\lambda,n)$ is a brick.
Since $End_A (M(b,\lambda,n)) \cong End_A (M(b,\lambda',n)) \cong K$, for all $\lambda, \lambda' \in K^*$, we have that $M(b,\lambda',n)$ is a brick for all $\lambda' \in K^*$.
In particular, this implies that there is an infinite number of bricks in $\mod A$
and by \cite{DIJ19} $A$ is $\tau$-tilting infinite.
\end{proof}

Motivated by Proposition~\ref{prop:bandbrick}, we show some general results on endomorphisms of band modules. 

\begin{theorem}\label{thm:lengthgraphmap}
Let $A$ a finite dimensional algebra and let $b$ be a band.
Suppose that there exists a  non-trivial nilpotent endomorphism $f_w : M(b,\lambda, 1) \to M(b, \lambda, 1)$  induced by the subword $w$ of $^\infty b ^\infty$.
Then $w$ is a proper subword of some rotation of $b$.
\end{theorem}

\begin{proof}
Since $f_w$ is a non-trivial nilpotent endomorphism of $M(b,\lambda, 1)$, we have that $w$ is both a quotient substring and a submodule substring of $^\infty b ^\infty$.
This is equivalent to the existence of letters $\alpha, \overline{\beta}, \overline{\gamma}, {\delta} \in Q_1 \cup \overline{Q}_1$ such that $\alpha w \overline{\beta}$ and $\overline{\gamma} w \delta$ are subwords of $^\infty b ^\infty$.

Suppose the existence of a positive integer $n$ such that $w$ can be written as $w=(b')^n w'$ for some rotation $b'$ of $b$, where the length of $w'$ is strictly smaller than the length of $b$.
Since $w'$ is a subword of $b$, there exists a rotation $b''$ of $b$ such that $w=w'(b'')^n$. 
Hence $w=w' w'' w'$ for some subword $w''$ of $^\infty b ^\infty$ and we have $w'w''=(b')^n$ and $w''w'=(b'')^n$.

From the fact that $\alpha w \overline{\beta} = \alpha w' w'' w' \overline{\beta} = \alpha b^n w' \overline{\beta}$ is a subword of $^\infty b ^\infty$ we can conclude that the last letter of $w''$ is $\alpha$.
But, at the same time, from $\overline{\gamma} w \delta = \overline{\gamma} w' w'' w' \delta = \overline{\gamma} b^n w' \delta$ we can conclude that the last letter of $w''$ is $\overline{\gamma}$, a contradiction.
Thus the length of $w$ cannot be greater than the length of $b$.
\end{proof}

\begin{proposition}\label{prop:onlyonce}
Let $b$ be a band and suppose that there exists a  non-trivial nilpotent endomorphism $f_w : M(b,\lambda, 1) \to M(b, \lambda, 1)$  induced by the subword $w$ of $^\infty b ^\infty$.
If $\alpha$ is a letter of $w$ then $\alpha$ appears at least twice in $b$.
\end{proposition}

\begin{proof}
Suppose that $b=\alpha_1\alpha_2 \dots \alpha_t$.
By Theorem~\ref{thm:lengthgraphmap} and without loss of generality we can assume that $b = wu$, where the first letter of $u$ is direct and the last letter of $u$ is inverse, for some subword $u$ of $b$ of length at least one and that this occurrence of $w$ in $b$ is a quotient substring of $^\infty b ^\infty$.
Since $b$ starts with $w$, there exists a $r < t$ such that $w=\alpha_1 \alpha_2 \dots \alpha_r$.

On the other hand, $w$ is also a submodule string of $b$ and by Theorem~\ref{thm:lengthgraphmap} there exist strings  $v$ and $v'$ such that $b^2 = vwv'$, where the last letter of $v$ is direct and the first letter of  $v'$ is inverse. 
We choose $v$ to be the shortest possible, which implies that the length $s$ of $v$ is strictly less than $t$. 
Furthermore, $s \geq 1$ since the last letter of $v$ is direct, while the last letter of $u$ is inverse.

Now $w=\alpha_{s+1}\alpha_{s+2}\dots\alpha_{s+r}$.
Hence $\alpha_i=\alpha_{s+i}$ for all $1\leq i \leq r$.
If $s+r < t$, the result follows at once.
Otherwise, $\alpha_i=\alpha_{s+i}$ if $s+i \leq t$ and $\alpha_i=\alpha_{s+i-t}$ if $s+i>t$.
In particular, since $0<s<t$, $\alpha_i$ and $\alpha_{s+i-t}$ are distinct letters of $b=\alpha_1\alpha_2\dots\alpha_t$.
\end{proof}

\begin{proposition}\label{prop:sirectsummand} 
Let $w$ be a string in $A$ containing a substring $\alpha_1 \ldots \alpha_k \alpha_1$ such that $b=\alpha_1 \ldots \alpha_k$ is a band. 
Then  $\emph{top} (M(b,\lambda, 1))$ is a direct summand of  $\emph{top} (M(w))$, and  $\emph{soc} (M(b,\lambda, 1))$ is a direct summand of  $\emph{soc} (M(w))$, for every $\lambda\in K^*$.
\end{proposition}

\begin{proof}
Up to cyclic permutation $b$ can be written as $b=w_1 w_2 \dots w_{2t-1} w_{2t}$ where $w_{2i-1}$ is a direct string and $w_{2i}$ is an inverse string for all $1 \leq i \leq t$, for some positive integer $t$.
Then $\soc (M(b,\lambda, 1))$ is the direct sum of the simple modules $S(t(w_{2i-1}))=S(s(w_{2i}))$ for all $1 \leq i \leq t$ and $\top (M(b,\lambda, 1))$ is the direct sum of the simple modules $S(t(w_{2i}))=S(s(w_{2i+1}))$ for all $0 \leq i \leq t-1$.

If $\alpha_1$ is not the first letter of one of the $w_i$, then the result follows. 
Now suppose that $\alpha_1$ is the first letter of  $w_{i}$ for some $0 \leq i \leq 2t$. Without loss of generality assume that $\alpha_1$ is the first letter of the direct string $w_1$. 
Then $\alpha_k$ is the last letter of the inverse string $w_{2t}$. 
It immediately follows that  $\soc (M(b,\lambda, 1))$ is a direct summand of $\soc (M(w))$. Finally, to prove that $\top (M(b,\lambda, 1))$ is a direct summand of $\top (M(w))$, it is enough to observe that $S(s(w_{1}))$  is a direct summand of $\top (M(w))$, corresponding to the substring $\alpha_k\alpha_1$.
\end{proof}

\begin{proposition}\label{prop:toptosoc}
Let $b$ be a band in $A$ with no repeated letters. 
Then for all $\lambda\in K^*$, every non-trivial nilpotent endomorphism $f\in \emph{End}_A(M(b,\lambda,1))$ induces a map 
$$\overline{f}: \emph{top}(M(b,\lambda,1)) \to \emph{soc}(M(b,\lambda,1)).$$
In particular, the image of every non-trivial nilpotent endomorphism $f$ is semisimple.
\end{proposition}

\begin{proof}
Let $f \in \End_A(M(b,\lambda,1))$ be a non-zero endomorphism.
By Remark \ref{rmk:n=1} $f$ is given by a linear combination of morphisms of the form $f_w$ given by a submodule substring $w$ of $^\infty b^ \infty$ which is at the same time a quotient substring of $^\infty b^ \infty$. 
Since $b$ has no repeated letters, by Proposition \ref{prop:onlyonce} every $w$ corresponds to a vertex.
Therefore, every summand of $f$ is induced by a simple module which is a direct summand of both the top and socle of $M(b,\lambda,1)$. 
\end{proof}

\section{Bands and torsion classes}\label{sec:tors}
In this section we study torsion classes containing band modules. 
We show  that if a torsion class contains a band module which is not a brick then the torsion class contains all band modules in the same infinite family. Furthermore,  we show that if a band module $M(b, \lambda, 1)$ is a brick then, for any $\mu \in K^*$ with $\mu \neq \lambda$, the minimal torsion class containing $M(b, \lambda, 1)$ is distinct from the minimal torsion class containing $M(b, \mu, 1)$. More precisely, we show the following. 

\begin{theorem}\label{thm:torsion}
Let $A$ be a finite dimensional algebra containing a band module $M(b, \lambda, 1)$, for some $\lambda \in K^*$. 
\begin{enumerate}
    \item If $M(b, \lambda, 1) \in \T$ for some torsion class $\T$ then $M(b, \lambda, n) \in \T$, for all $n \in \mathbb{N}$.
    \item If $M(b, \lambda, 1)$ is not a brick  and $M(b, \lambda, 1)$ is in  some torsion class $\T$ then $M(b, \mu, 1) \in \T$ for all $\mu \in K^*$.
    \item If $M(b, \lambda, 1)$ is a brick then there exist  infinitely many  distinct torsion classes $\T_\mu$, for $\mu \in K^*$, such that $\T_\mu$ contains $M(b, \eta, 1)$, for $\eta \in K^*$, if and only if $\eta = \mu$.
\end{enumerate}
\end{theorem}

\begin{proof}
Let $M(b, \lambda, 1)$ be a band module in a torsion class $\T$.

(1) This follows from the fact that $\T$ is closed under extensions and $M(b,\lambda,n)$ is an extension of $M(b,\lambda, 1)$ by $M(b,\lambda, n-1)$.

(2) Suppose $M(b, \lambda, 1)$ is not a brick.
By Theorem~\ref{thm:lengthgraphmap},  there exists  a string module $M(w)$ which is at the same time a quotient and a submodule of $M(b, \lambda, 1)$ giving rise to a short exact sequence 
$$0 \longrightarrow M(w) \longrightarrow M(b, \lambda, 1) \longrightarrow N \longrightarrow 0$$
where $N$ is the string module $M(b, \lambda, 1)/M(w)$.
Both $M(w)$ and $N$ are quotients of $M(b, \lambda, 1)$ and thus since $\T$ is closed under quotients, we have $M(w) \in \T$ and $N \in \T$.

Moreover, we have that the string modules $M(w)$ and $N$ are quotients of $M(b, \mu, 1)$, for any $\mu \in K^*$, since  the construction of string modules $M(w)$ and $N$ is independent of the parameter $\lambda$. 
Moreover, for every $\mu \in K^*$ we have a short exact sequence 
$$0 \longrightarrow M(w) \longrightarrow M(b, \mu, 1) \longrightarrow N \longrightarrow 0.$$ 
Then the result follows from the fact that $\T$ is closed under extensions. 

(3) Recall that given a band $b$ such that $M(b,\lambda, 1)$ is a brick for some $\lambda \in K^*$, then $M(b,\mu, 1)$ is a brick for all $\mu \in K^*$.
Moreover, we have that, for $\eta \in K^*$,  $\Hom_A(M(b, \eta, 1), M(b, \mu, 1)) =0$ if and only if $\mu \neq \eta$. 
Define $\T_\mu$ to be $\rm{Filt} (\Fac (M(b,\mu,1)))$, the class of all $A$-modules filtered by quotients of an element in $add (M(b,\mu,1))$.
We claim that, for $\eta \in K^*$, $M(b,\eta, 1)$ is not in $\T_\mu$ if $\mu \neq \eta$.
Suppose to the contrary that $M(b,\eta,1) \in \T_\mu$.
Then, $M(b, \eta, 1)$ is filtered by objects in $\Fac(M(b,\mu,1))$.
Hence there is a submodule $0 \neq L$ of $M(b,\eta,1)$ which is in $\Fac(M(b,\mu,1))$ and there is a non-zero map from $M(b,\mu,1)$ to $M(b,\eta,1)$ with $\mu \neq \eta$, a contradiction.
\end{proof}

\begin{remark}
The infinite family of torsion classes in Theorem~\ref{thm:torsion}(3) is such that two torsion classes in that family are not comparable in the poset of torsion classes of $A$. 
\end{remark}

\begin{remark} 
The results in Theorem \ref{thm:torsion}(1) and (2) serve as an indication of why  $\tau$-tilting finite algebras of infinite and of even wild representation type  exist. Examples of $\tau$-tilting finite wild algebras are  preprojective algebras of Dynkin type with at least six vertices \cite{Miz14} and wild contraction algebras \cite{Aug18}.
\end{remark}


\section{$\tau$-tilting finiteness for special biserial algebras}

In this section we apply the results of Section \ref{sec:bands} to show that the converse of Proposition~\ref{prop:bandbrick} holds for special biserial algebras. 
More precisely, we show the following. 

\begin{theorem}\label{thm:specialbis}
Let $A=KQ/I$ be a special biserial algebra.
Then $A$ is $\tau$-tilting finite if and only if no  band module of $A$ is a brick.
\end{theorem}

Recall that an algebra $KQ/I$ is \emph{special biserial} if  every vertex in $Q$ is the start of at most two arrows, and the end of at most two arrows  and if  for every  $\alpha \in Q_1$ there is at most one $\beta \in Q_1$ such that $\alpha \beta \notin I$ and there is at most one $\gamma \in Q_1$ such that $\gamma \alpha \notin I$.

In order to  prove Theorem~\ref{thm:specialbis}, we first show the following lemma. 

\begin{lemma}\label{lem:bandsspecbiser}
Let $A=KQ/I$ be a special biserial algebra. 
Then all but finitely many strings $w$ contain a subword of the form $\alpha\overline{\beta} v \alpha$, for $\alpha, \beta \in Q_1$ and some string $v$,  such that $\overline{\beta}v\alpha$ is a band and $\alpha$ is not a letter of $v$.
\end{lemma}

\begin{proof}
Let $w$ be a string, $x$ be a vertex in $Q$ and $\alpha, \beta \in Q_1$ such that $t(\alpha)=t(\beta)=x$.
Since $A$ is special biserial, the simple module $S_x$ associated to the vertex $x$ is a direct summand of the socle $\soc(M(w))$ of the string module $M(w)$  if and only if either $w$ starts with $\overline{\alpha}$ or $\overline{\beta}$, if $w$ finishes with $\alpha$ or $\beta$, or if $\alpha\overline{\beta}$ or $\beta\overline{\alpha}$ is a subword of $w$.

Suppose that the dimension  of  $\soc(M(w))$  is greater or equal to $2n+3$, where $n$ is the number of non-isomorphic simple $A$-modules.
This implies that there are $2n+1$ simple direct summands of $\soc(M(w))$ given by a subword of the form $\gamma\overline{\delta}$ of $w$, for some $\gamma, \delta \in Q_1$.
Then, by the pigeon hole principle, without loss of generality we can suppose that  $\alpha\overline{\beta}$ appears at least twice as a subword of $w$.
Then there exist strings $v, w_1$ and $w_2$ such that  $w=w_1 \alpha\overline{\beta}v\alpha \overline{\beta} w_2$. Furthermore, we can assume that there is no occurrence of $\alpha$ in $v$.

We claim that $\overline{\beta}v\alpha$ is a band in $A$.
First, note that $s(\overline{\beta})=t(\alpha)$.
Moreover, $\overline{\beta}v\alpha$ starts with an inverse letter and finishes with a direct letter, so there are no subwords of any rotation $\overline{\beta}v\alpha$ which are in $I$. 
By construction the letter $\alpha$ appears only once in $\overline{\beta}v\alpha$. 
As a consequence, $\overline{\beta}v\alpha$ is not the concatenation of several copies of a smaller band.
Hence, $\overline{\beta}v\alpha$ is a band. 

Since the quiver has finitely many arrows and $A$ is finite dimensional, there are only finitely many strings $w$ such that $\dim_K(\soc(M(w)))\leq 2n+3$. 
\end{proof}

\begin{proof}[Proof of Theorem~\ref{thm:specialbis}]
Suppose that no band module in the module category of $A$ is a brick. 
Then any brick must be  a string. 
We claim that there are only finitely many string modules which are bricks.

By Lemma \ref{lem:bandsspecbiser} we have that all but finitely many strings $w$ are of the form $w=w_1\alpha \overline{\beta} v \alpha w_2$ where $b=\overline{\beta} v \alpha$ is a band in $A$.
We claim that the string module $M(w)$ of any string $w$ of this form is not a brick. 
By hypothesis, $M(b,\lambda, 1)$ is not a brick.
So, there exists a non-trivial nilpotent endomorphism $f_{w'}: M(b,\lambda, 1) \to M(b,\lambda,1)$ which is realised by a substring $w'$ of $^\infty b ^\infty$ by Remark\;\ref{rmk:n=1}.
Now, by Theorem \ref{thm:lengthgraphmap}  the submodule substring and the quotient substring $w'$ of $b$ are included in some rotation of $b=\overline{\beta}v\alpha$. 
But Proposition \ref{prop:onlyonce} implies that $\alpha$ is not a letter of $w'$ since $\alpha$ only appears once in $b$.
Hence, both copies of $w'$ are subwords of $\overline{\beta}v$.
As a consequence, $w'$ is both a quotient and a submodule substring of $w$.
Thus, there is a non-trivial nilpotent endomorphism $f_{w'}: M(w) \to M(w)$ that is realised by the substring $w'$ of $w$.
In other words, $M(w)$ not a brick.

The converse follows from Proposition~\ref{prop:bandbrick}.
\end{proof}

As a consequence of Theorem \ref{thm:specialbis} we obtain a proof of both $\tau$-Brauer-Thrall conjectures for special biserial algebras. 

\begin{corollary}
Conjectures \ref{conj:tBT1} and \ref{conj:tBT2} hold for special biserial algebras.
\end{corollary}

\begin{proof}
We show Conjecture \ref{conj:tBT1} by contraposition.
Suppose that a special biserial algebra $A$ is $\tau$-tilting infinite.
By Theorem~\ref{thm:specialbis} there exists a band module $M(b, \lambda, 1)$ which is a brick.
Hence there is no subword of $^\infty b ^\infty$ which is both a submodule substring and a quotient substring.
Let $b= \alpha_1 \alpha_2 \dots \alpha_t$ where $\alpha_1$ is a direct letter and $\alpha_t$ is an inverse letter.
Then, for every positive integer $n$, every proper quotient substring of $b^n$ is a quotient substring of $^\infty b ^\infty$.
Likewise, for every positive integer $n$, every proper submodule substring of $b^n$ is a submodule substring of $^\infty b ^\infty$.
Then, it follows from \cite{CB89} that $M(b^n)$ is a brick for all positive integer $n$.
Thus Conjecture~\ref{conj:tBT1} holds for special biserial algebras.

Suppose that $A$ is special biserial $\tau$-tilting infinite.
Again, by Theorem~\ref{thm:specialbis} there exists a band module $M(b, \lambda, 1)$ which is a brick.
Then, for any $\mu \in K^*$, $M(b, \mu, 1)$ is a brick.
Since $K$ is an algebraically closed field, Conjecture~\ref{conj:tBT2} follows.
\end{proof}


\section{Characterisation of $\tau$-tilting finite Brauer Graph Algebras}\label{sec:BGA}

An algebra $A$ is said to be \textit{symmetric} if $A \cong \Hom_A(A,k)$ as $A$-$A$-bimodule. 
It directly follows from the definition of a Brauer graph algebra that it is special biserial and it was shown  in \cite{ Roggenkamp, Schroll15}, see also \cite{Antipov}, that every symmetric special biserial algebra is a  \textit{Brauer graph algebra}.
A \textit{Brauer graph} is a finite undirected connected graph, possibly with multiple edges and loops, in which every vertex is equipped with a cyclic ordering of the edges incident with it and a strictly positive integer, its \emph{multiplicity}. 

 We briefly recall here that given a Brauer graph the corresponding Brauer graph algebra is given by a quotient of a path algebra of a quiver $Q$ by an ideal $I$ generated by relations. The vertices in $Q$ are in bijection with the edges of the Brauer graph and the arrows of $Q$ are induced by the cyclic orderings of the edges at each vertex of the Brauer graph. Each vertex in the Brauer graph gives rise to a cycle in the quiver, sometimes referred to as a special cycle. The relations of the Brauer graph algebra can be read directly from the special cycles and the Brauer graph.  For a precise definition and the  construction of a symmetric special biserial algebra from a Brauer graph we refer the reader to standard textbooks such as \cite{Ben98}.

Before we start, we need to fix some notation.
A \textit{cycle} $C$ in a Brauer graph $G$ is a set of vertices $\{v_1, \dots, v_n\}$ and a set of edges $\{e_1, \dots, e_{n}\}$ such that $e_n$ is incident with $v_1$ and $v_n$, and $e_i$ is incident with $v_{i-1}$ and $v_i$ for all $1\leq i \leq n-1$.
A  cycle $C$ is \textit{minimal} if all its vertices are distinct.
We say that a  cycle $C$ is \textit{odd} (respectively, \textit{even}) if it is a minimal cycle with an odd (respectively, even) number of vertices.

As  an application of Theorem~\ref{thm:specialbis}, we give a new proof of the characterisation in \cite{AAC18} of the  $\tau$-tilting finiteness of Brauer graph algebras in terms of their Brauer graph.  
More precisely, we show the following. 

\begin{theorem}\cite[Theorem~6.7]{AAC18}\label{thm:BGA}
Let $A=KQ/I$ be a Brauer graph algebra with Brauer graph $G$. 
Then $A$ is $\tau$-tilting finite if and only if $G$ has no even cycles and  at most one odd cycle.
\end{theorem}

In order to show Theorem \ref{thm:BGA}, we first show the following lemmas.

\begin{lemma}\label{lem:even}
Let $A=KQ/I$ be a Brauer graph algebra with Brauer graph $G$. 
If  $G$ has an even cycle, then there is a band $b$ such that $M(b,\lambda, 1)$ is a brick.
\end{lemma}

\begin{proof}
By hypothesis $G$ has  an even cycle $C$ with pairwise distinct vertices  $v_1, \dots, v_{2t}$ and edges $e_1, \dots, e_{2t}$ in $G$ such that $e_i$ is incident with $v_i$ and $v_{i+1}$ for all $i$ and $e_{2t}$ is incident with $v_{2t}$ and $v_1$.
Now, define $w_i$ to be the shortest direct path in $Q$ from $e_i$ to $e_{i+1}$ if $i$ is even and the shortest inverse path from $e_{i}$ to $e_{i+1}$ if $i$ is odd. 
Then it is easy to see that the word $w = w_1w_2 \dots w_{2t}$ is a band.

We claim that $M(w,\lambda,1)$ is a brick.
Indeed, by construction  $w$ has no repeated letters. 
Then, Proposition \ref{prop:toptosoc} implies that every non-trivial nilpotent endomorphism $f$ of $M(w,\lambda,1)$ factors through a map $\overline{f}: top(M(w,\lambda,1)) \to soc(M(w,\lambda,1))$.
By construction, we have that 
$$top(M(w,\lambda,1)) \cong \bigoplus_{j=1}^t S(e_{2j}) \quad \text{ and } \quad soc(M(w,\lambda,1)) \cong \bigoplus_{j=1}^t S(e_{2j-1})$$ where all the $S(e_i)$ are distinct.  
Then $M(w,\lambda,1)$ has no non-trivial nilpotent endomorphisms and thus it is a brick. 
\end{proof}

\begin{remark}
If a Brauer graph $G$ is not simply laced then it has a cycle of length 2 and it follows from the previous lemma that it contains a band module which is a brick. 
\end{remark}

\begin{lemma}\label{lem:twooddsmakeaneven}
Let $A=KQ/I$ be a Brauer graph algebra with Brauer graph $G$. 
If $G$ has two odd cycles, then there is a band $b$ such that $M(b,\lambda, 1)$ is a brick.
\end{lemma}

\begin{proof}
Suppose that $G$ has two odd cycles. 
Then there exist two sets of vertices, $v_1, \dots, v_{2t+1}$ and $v'_1, \dots, v'_{2r+1}$, and two set of edges $e_1, \dots, e_{2t+1}$ and $e'_1, \dots, e'_{2r+1}$ in $G$ such that $e_i$ is incident with $v_i$ and $v_{i+1}$ for all $i$ and $e'_j$ is incident with $v'_j$ and $v'_{j+1}$, $e_{2t+1}$ is incident with $v_{2t+1}$ and $v_1$ and $e'_{2r+1}$ is incident with $v'_{2r+1}$ and $v'_1$.
Since $G$ is connected, there exists a set of vertices $v''_1, \dots, v''_s$ and edges $e''_1, \dots, e''_{s+1}$ such that $e''_1$ is incident with $v_1$ and $v''_1$, $e''_{s+1}$ is incident with $v''_s$ and $v'_1$ and $e''_k$ is incident with $v''_{k-1}$ and $v''_{k}$ for all $2\leq k \leq s$.
Note that if $v_1=v'_1$, then $k=0$.
This proof consist of two cases: when $k$ is odd and when $k$ is even.
We prove the case of $k$ odd, the case of $k$ even being very similar.

Similarly to the proof of Lemma \ref{lem:even}, we  construct a suitable band in $Q$ and we will do so in several steps. 

First, define $w_i$ to be the shortest direct path from $e_i$ to $e_{i+1}$ if $i$ is even and the shortest inverse path from $e_{i}$ to $e_{i+1}$ if $i$ is odd for all $1 \leq i \leq 2t$.
Now let $w_{2t+1}$  be the shortest direct path from $e_{2t+1}$ to $e''_1$.
Denote by $u_{i}$ the shortest direct path from $e''_{i}$ to $e''_{i+1}$ if $i$ is odd and the shortest inverse path from $e''_{i}$ to $e''_{i+1}$ if $i$ is even for all $i$ between $0$ and $s-1$.
Define $w'_{0}$ as the shortest direct path from $e''_{s+1}$ to $e'_1$.
For all $1 \leq j \leq 2r$ define $w'_{j}$ to be the shortest direct path from $e'_j$ to $e'_{j+1}$ if $j$ is even and the shortest inverse path from $e'_{j}$ to $e'_{j+1}$ if $j$ is odd.
Set $w'_{2r+1}$ as the shortest inverse path from $e'_{2r+1}$ to $e''_{s+1}$.
Finally let $w_0$ be the shortest direct path from $e''_1$ to $e_1$.

By construction $w=w_1 \ldots w_{2t+1} u_0 \ldots u_{s-1} w'_0 \ldots w'_{2r+1} \overline{u_{s-1}} \ldots \overline{u'_0} w_0$ is a band  
and $w$ has no repeated letters.
Then, Proposition \ref{prop:toptosoc} implies that every non-trivial nilpotent endomorphism $f$ of $M(w,\lambda,1)$ factors through a map $\overline{f}: \top(M(w,\lambda,1)) \to \soc(M(w,\lambda,1))$.
Furthermore, by construction, we have that $\top(M(w,\lambda,1))$ and $\soc(M(w,\lambda,1))$ have no-common direct summand, thus implying that $M(w,\lambda,1)$ has no non-trivial nilpotent endomorphisms. 
In other words, $M(w,\lambda,1)$ is a brick in $\mod A$.
\end{proof}

We now prove Theorem \ref{thm:BGA}.

\begin{proof}[Proof of Theorem \ref{thm:BGA}]
Let $A$ be a Brauer Graph algebra with Brauer graph $G$.  Recall that every indecomposable non-projective $A$-module  $M$ comes from a string or a band in the Brauer graph. Furthermore, 
if $G$ contains an even cycle or if $G$ contains two odd cycles then by Lemmas~\ref{lem:even} and \ref{lem:twooddsmakeaneven} the algebra  $A$ is $\tau$-tilting infinite. Thus suppose that $G$ contains at most one odd cycle. 

If $G$ is a tree and all but at most one multiplicity is equal to one then $A$ is of finite representation type and,
in particular, $A$ is $\tau$-tilting finite.

Now suppose that $G$ is a tree and that there are at least two vertices of $G$ with multiplicity strictly greater than one and let $b$ be a band in $A$.
Given that $G$ is a tree, there exists a vertex $v$ in $G$ with multiplicity strictly greater than one such that $b= w b'$, where $w$ is a direct or inverse path maximal in $b$ starting and ending at the same  edge $x$ of $G$ which is incident with $v$. 
By maximality of $w$, we have that the simple module $S(x)$ associated to $x$ is a direct summand of both $\soc(M(b,\lambda, 1))$ and $\top(M(b,\lambda, 1))$, for any $\lambda \in K^*$.
Hence no band module in $\mod A$ is a brick. 
So, $A$ is $\tau$-tilting finite by Theorem \ref{thm:specialbis}.

The last case to consider is when $G$ has exactly one cycle of odd length $2t+1$. 
Then there exists a set of vertices $v_1, \dots, v_{2t+1}$ and edges $e_1, \dots, e_{2t+1}$ in $G$ such that $e_i$ is incident with $v_i$ and $v_{i+1}$ for all $i$ and $e_{2t+1}$ is incident with $v_{2t+1}$ and $v_1$.
Since $G$ has no even cycle, it is simply-laced and $e_i$ is the unique edge incident with $v_i$ and  $v_{i+1}$. 

Now consider a band $b$ in $A$. 
If $b$ is such that there exists a vertex $v$ such that $b = w b'$ as in the case of the tree with at least two vertices of of higher multiplicities, then $M(b,\lambda, 1)$ is not a brick.

Otherwise $b$ is of the form  $b= w_1 \dots w_{2t+1} u_1 \dots u_{2t+1}$, where $w_i$ is a direct path from $e_i$ to $e_{i+1}$ if $i$ is even and the inverse path from $e_{i}$ to $e_{i+1}$ if $i$ is odd for all $1 \leq i \leq 2t+1$ and $u_{i}$ is a the inverse path from $e_i$ to $e_{i+1}$ if $i$ is even and the direct path from $e_{i}$ to $e_{i+1}$ if $i$ is odd for all $1 \leq i \leq 2t+1$.
Then the simple module $S(e_i)$ is a direct summand of both $\top(M(w,\lambda, 1))$ and $\soc(M(w,\lambda, 1))$.
So, $M(w,\lambda, 1)$ is not a brick and by the same argument as above, $A$ is $\tau$-tilting finite.
\end{proof}

\end{document}